\documentclass{amsart}
\usepackage{amsfonts}
\usepackage{amsmath}
\usepackage{amssymb}
\usepackage{graphicx}
\usepackage{version}
\usepackage{color}
\usepackage{caption}

\theoremstyle{plain}
\newtheorem{theorem}{Theorem}
\newtheorem{cor}{Corollary}
\newtheorem{lemma}{Lemma}
\newtheorem{proposition}{Proposition}[section]
\theoremstyle{definition}
\newtheorem{definition}{Definition}

\newtheorem{remark}{Remark}

\numberwithin{equation}{section}
\newcommand{\R}{{\mathbb R}}

\begin{document}

\title{Distortion Reversal in Aperiodic Tilings}
\author[L. Barnsley]{Louisa Barnsley} 
\address{Mathematical Sciences Institute \\ Australian National University \\ Canberra,
ACT, Australia}
\email{louisabarnsley@gmail.com}

\author[M. Barnsley]{Michael Barnsley}
\address{Mathematical Sciences Institute \\ Australian National University \\ Canberra,
ACT, Australia}
\email {michael.barnsley@anu.edu.au}

\author[A. Vince]{Andrew Vince}
\address{Department of Mathematics \\ University of Florida \\ Gainesville, FL, USA}
\email{avince@ufl.edu}

\thanks{This work was partially supported by a grant from the Simons Foundation (\#322515 to Andrew Vince).}


\begin{abstract}
It is proved that homeomorphic images of certain two-dimensional aperiodic
tilings, such as Ammann-A2 tilings, are recognizable, in both mathematical and practical senses. One
implication of the results is that it is possible to search for distorted aperiodic
structures in nature, where they may be hiding in plain sight.
\end{abstract}

\maketitle

\section{Introduction}

Herein we prove that homeomorphic images of certain two-dimensional self-similar
aperiodic tilings are recognizable, in both mathematical and practical
senses. In the practical sense our results say it is possible to search for
and recognize distorted aperiodic structures, for example in natural and physical settings. 
We recall that the mathematical discovery
of aperiodic tilings by Penrose in 1972 and Ammann in 1975
preceded the discovery by Shechtman in 1982 of quasicrystals, and the subsequent examples of 
quasicrystals in meteorites in 2009 \cite{steinhardt}.

In this paper we focus on Ammann-A2 tilings \cite{ammann}, but the ideas are more generally applicable.
Our results illustrate that aperiodic tilings can be retrieved from their images distorted by \textit{unknown} homeomorphisms. 
For example, we show how distorted Ammann-A2 tilings, such as the ones in Figure \ref{four}, are 
recognizable both practically and mathematically. In Figure~\ref{(dents_seq2x2)} an algorithm (without knowledge of the
distorting homeomorphism) is applied to the distorted tiling to successively retrieve a patch of the original Ammann-A2 tiling.  

The situation is a bit strange: the unknown distortion may be such as to change the Hausdorff dimension 
of the boundaries of the tiles in a very rough way, or to transform an Ammann-A2 tiling to a tiling by triangles.
Nevertheless, the key properties of the tiling may still be discernable. 
\vskip 2mm

\begin{figure}[htb]
\centering
\includegraphics[width=\textwidth]{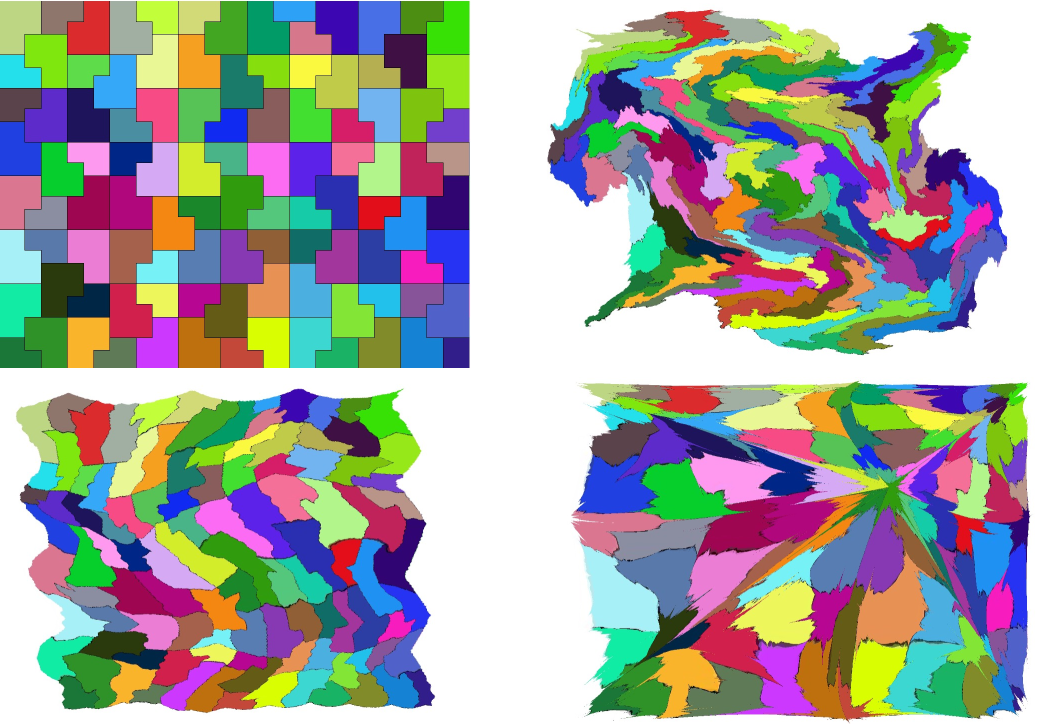}
\caption{This illustrates a patch of an Ammann-A2 tiling, top left, and three different homeomorphisms of it. 
Properties of the Ammann-A2 tiling may be determined from a distorted image.}
\label{four}
\end{figure}

There is a substantial literature on the occurrence of two-dimensional
tilings in physics \cite{weaire}. Many of these are distortions of tilings
by hexagons. Many naturally occurring tilings have an average of six faces per tile, and typical tiles are hexagons. 
Ammann-A2 tiles are hexagonal and comprise among the simplest of self-similar tilings. Despite this,
 they do not seem to show up in nature. Are they hiding in plain sight?  We provide some tools which may 
be applied to this question.

\section{Organization and Main Results}

Let $P$ be a finite aperiodic prototile set.  This means that there exist tilings of the plane by copies 
of the tiles in $P$, but every such tiling is non-periodic, i.e., has no translational symmetry.  Such
tilings have been objects of fascination, both in research and recreational mathematics, since
their introduction in the 1960's.  The most well-known such prototile set consists of 
the two Penrose tiles. To develop the ideas in this paper we 
concentrate, because of their simplicity, on tilings based on the set of two Ammann-A2 tiles.  There
are uncountably many corresponding Ammann-A2 tilings, obtained either  
 by matching rules or by substitution rules.  
Basic properties of such Ammann-A2 tilings are provided in Section~\ref{sec:a2}.
\vskip 2mm

\begin{figure}[htb]
\centering\includegraphics[width=\textwidth]{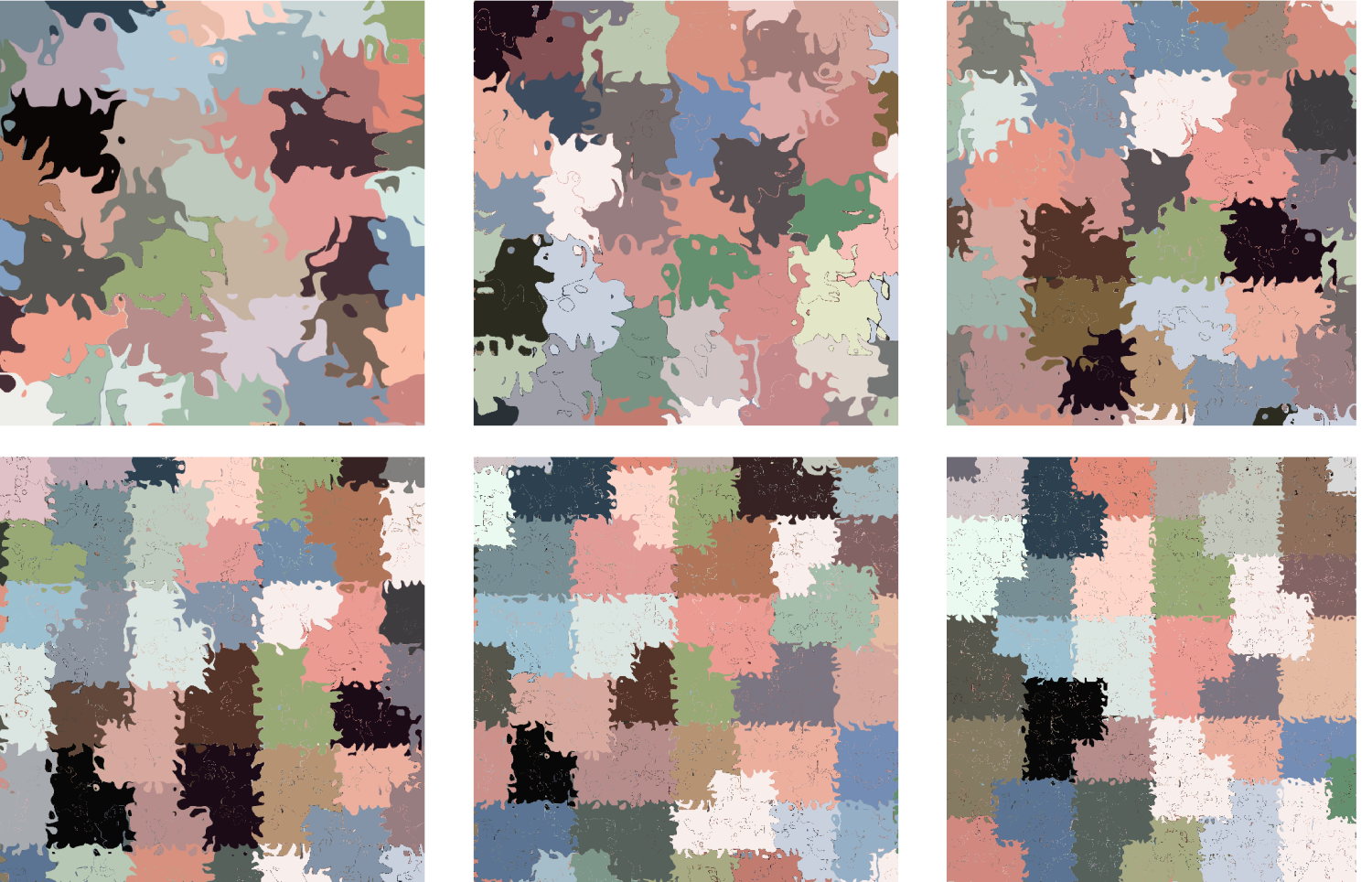}
\caption{Successive images, obtained by amalgamation, converge towards patches of an Ammann-A2 tiling.}
\label{(dents_seq2x2)}
\end{figure}

Many non-periodic tilings have a hierarchical structure, meaning that if $T$
is such a tiling, then there exist a sequence of tilings $T = \overline T_0,  \overline T_1,  \overline T_2, \dots$,
 such that, for all $n\geq 1$, each tiling in $\overline T_n$ is the non-overlapping union of tiles in $\overline T_{n-1}$.
For the Ammann-A2 tilings, the tiling $\overline T_n$, appropriately scaled, is again an 
Ammann-A2 tiling.  Specifically, there is a constant $s$ such that $T_n := s^n \overline T_n$ is an
Ammann-A2 tiling for all $n$.  There is a well-defined procedure, which we call {\it amalgamation},
taking $T_n$ to $T_{n+1}$, and denoted $T_{n+1} = a(T_n)$.  
 Amalgamation and hierarchy for the Ammann-A2 tilings
is explored in Section~\ref{sec:A}.
\vskip 2mm

Here is the question that is the subject of this paper.  Suppose that we are
given a distorted version $T'$ of an Ammann-A2 tiling $T$.  Specifically, there 
is an unknown (not revealed to us) homeomorphism  $h$  of the plane such that 
$T' = h(T)$.  Is it possible to retrieve the undistorted Ammann-A2 tiling $T$ from its
distorted image $T'$ - without knowledge of $h$?  As long as the
distance between any point $x$ in the plane and its image $h(x)$ is bounded,
we answer this question in the affirmative.  Suppose that amalgamation 
can be carried over to the distorted tiling $T'$. Then 
Section~\ref{sec:bdh} contains the following somewhat surprising result:  if $T'_n = a^n(T'), \, n\geq 0,$
are the tilings obtain by successive application of the amalgamation operator $a$ applied to the distorted image $T'$,
then the sequence $\{T_n\}$ of tilings
converges, in the sense of Theorem~\ref{thm:m} in Section~\ref{sec:bdh} and Remark~\ref{rem:converge}, to the original tiling $T$.  
This is illustrated in Figure~\ref{(dents_seq2x2)}.
\vskip 2mm

The above distortion reversal result is predicated on being able to amalgamate tiles of the
distorted tiling $T'$, a tiling in which metric and geometric properties of the original tiling are lost.
Only the combinatorial properties of $T$ are carried over to $T'$.  
In Section~\ref{sec:ca} we obtain, completely combinatorially, the needed amalgamation operator $a$ applied to $T'$.  
This is done in Algorithm A in Section~\ref{sec:ca}; see Theorem~\ref{thm:ed}. Of course,
the plane and tilings of it, are unbounded. We would like to perform the distortion
reversal efficiently on a large patch of the tiling $T'$, 
say on a patch that contains the disk $D_R$
of radius $R$ centered at the origin.  Algorithm B, a variation of Algorithm A, accomplishes this.  
Theorem~\ref{thm:ed2} in Section~\ref{sec:ca} states that 
each application of the combinatorial amalgamation applied to such a patch can be performed in
running time $O(R^2)$.  
\vskip 2mm

An alternative method for reversing the distortion in an Ammann-A2 tiling is
provided in Section~\ref{sec:dr}.  It uses the combinatorial amalgamation operator applied to the distorted tiling $T'$ to
generate an infinite binary string $c$, called the code of the tiling.  The code is the input for an iterated function system
based method to retrieve the original tiling $T$.  This method, encapsulated in Theorem~\ref{thm:c} of Section~\ref{sec:dr}, 
has the  advantage that there is no restriction on the homeomorphism $h$.  It has the disadvantage that there is no
sequence of increasingly accurate images converging to the original tiling $T$; it simply produces $T$ or, in finite time, 
an arbitrarily large patch of $T$.

\section{Ammann-A2 Tilings}

\label{sec:a2}

In this paper a \textit{tile} is defined as a set in the plane homeomorphic to
a closed disk. A \textit{tiling} of the plane is a set of non-overlapping
tiles whose union is ${\mathbb{R}}^{2}$. By \textit{non-overlapping} is meant
that the intersection of any two distinct tiles has empty interior. For all
tilings in this paper, the intersection of any two distinct tiles is connected
(possibly empty).  A {\it patch} of a tiling is a finite number of 
tiles whose union is a topological disk.

The classical Ammann-A2 hexagon $G$, sometimes referred as the
\textit{golden bee}, is depicted in Figure~\ref{figa1}. It is the only
polygon, other than  any right triangle and any parallelogram
with side lengths in the ratio $\sqrt{2}:1$, that is the non-overlapping union
of two smaller similar copies of itself \cite{schmerl}. These two smaller
copies are isometric to $sG$ and $s^{2}G$, where $s=1/\tau$ and $\tau$ is the
square root of the golden ratio, i.e., $\tau$ is the positive real root of
$x^{4}-x^{2}-1=0$. 

\begin{figure}[tbh]
\centering
\includegraphics[height=4.1474cm,width=3.31185cm]{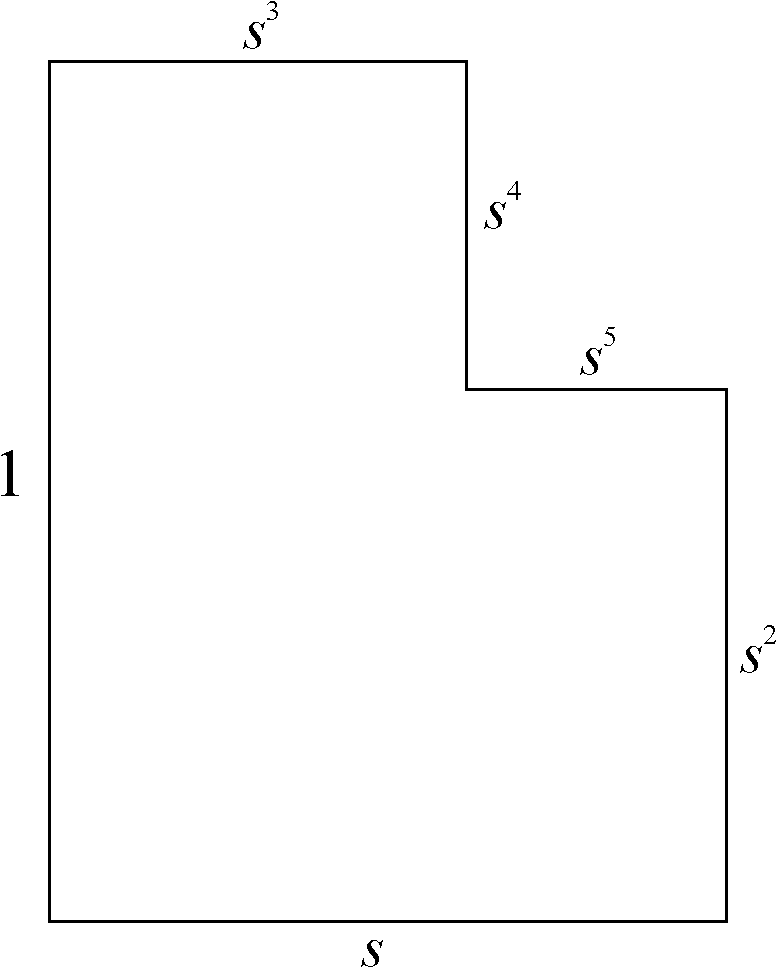} \hskip 10mm
\includegraphics[width=3.5cm, keepaspectratio]{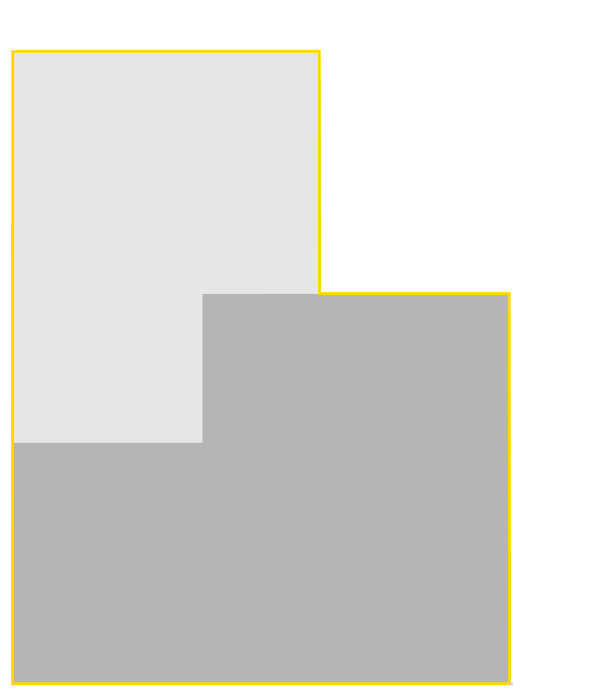}
\caption{Golden bee $G$}
\label{figa1}
\end{figure}

An \textit{Ammann-A2 tiling} is a tiling of the plane by non-overlapping
isometric copies of  $sG$ and $s^{2}G$, which we will refer to
as the \textit{large} and \textit{small} tiles, respectively.
The tiling must obey matching rules dictated by the
elliptical markings on the upper tiles of the left panel in Figure~\ref{fig:figA2X}.
Other alternative but equivalent sets of
markings are also possible, for example the markings shown in the right panel
of Figure~\ref{fig:figA2X}.  Although the decorations on small and large tiles at the top of the left panel and the decorations on the small and large tiles of the right panel differ, matching according to either set of decorations define the Ammann-A2 tilings, as observed in 
 \cite[p.~551]{grunbaum}; see also \cite{akiyama, ammann, durand, korotin}.

\begin{figure}[hbt]
\centering
\includegraphics[width=4.8cm]{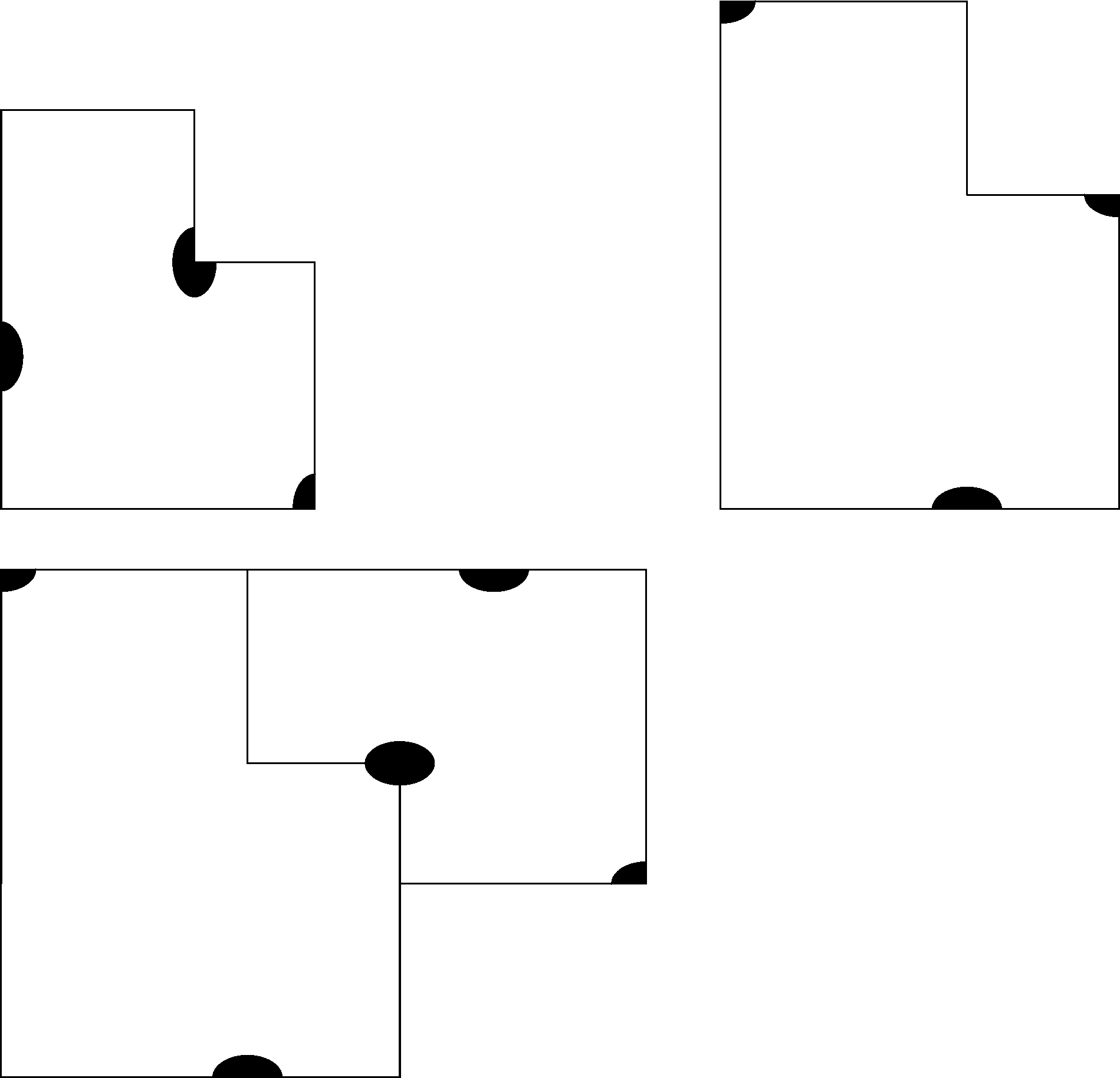}\hskip 15mm
\includegraphics[width=4.8cm]{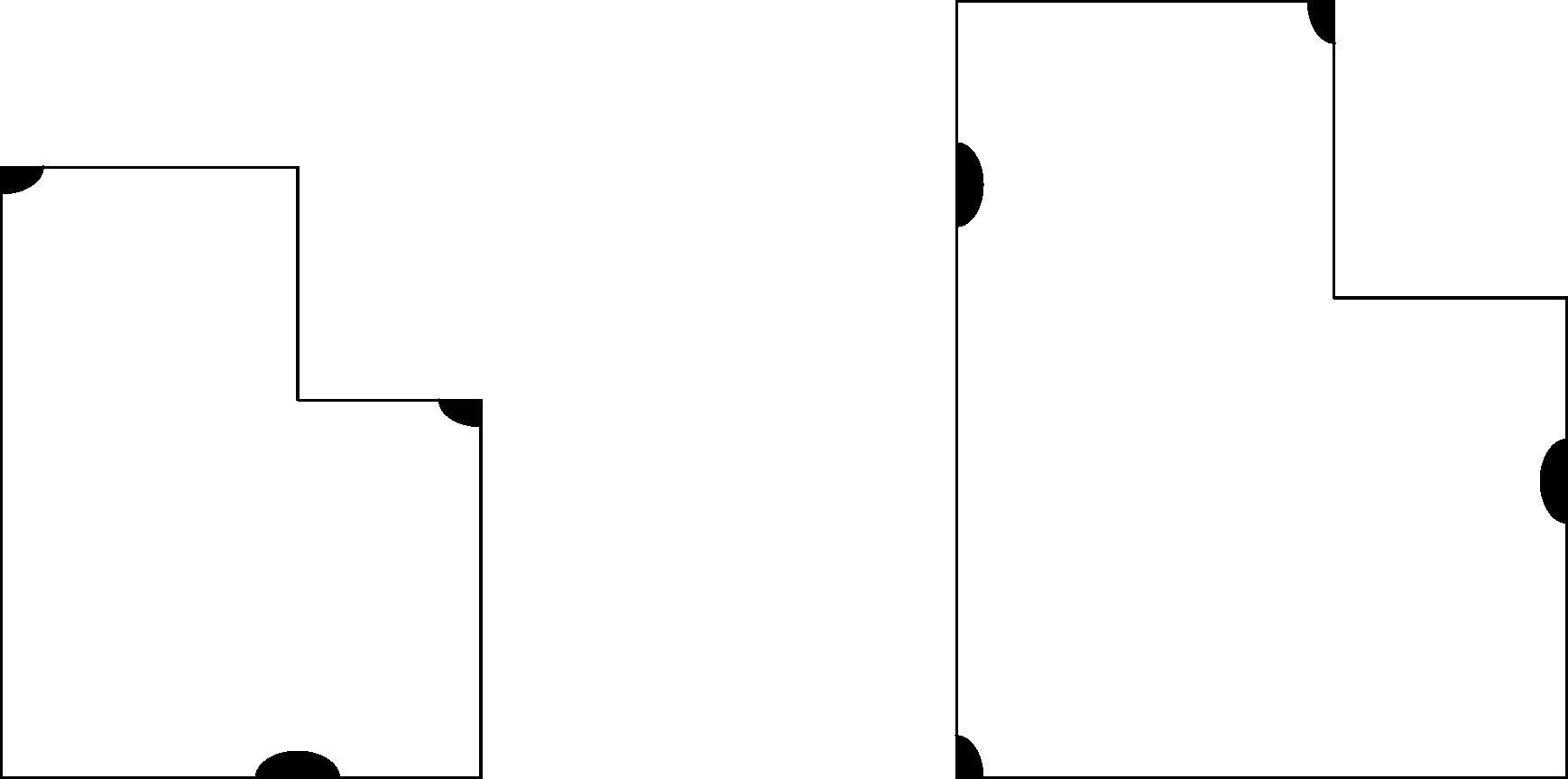}
\vskip 2mm
\caption{Two equivalent matching rules for the Amman tilings.  As explained in the paragraph above, the
matching rules dictated by the decorated pair of small and large tiles at the upper left and the matching
rules dictated by the decorated pair of small and large tiles at the right both define the Ammann tilings.   
As explained in Section~\ref{sec:A}, the figure at the lower left shows 
how a small Ammann tile meets its large partner.  }
\label{fig:figA2X}
\end{figure}

Hereafter we will call an Ammann-A2 tiling simply an {\it Ammann tiling}. A portion of an Amman
tiling is shown in Figure~\ref{figa3}. Two tilings $T$ and $T^{\prime}$ are
said to be \textit{congruent} if there is an isometry
$U$ such that $T^{\prime}=U(T)$. Much has been written on Ammann tilings both in
mathematics journals, for example \cite{durand,grunbaum} and in recreational
sources, notably \cite{gardner,senechal}. There are uncountable many Ammann
tilings, each being non-periodic. Every Ammann tiling is repetitive and every
pair of Ammann tilings are locally isomorphic. A tiling $T$ is
\textit{repetitive}, also called \textit{quasiperiodic}, if, for every finite
patch $P$ of $T$, there is a real number $R$ such that every ball of radius
$R$ contains a patch congruent to $P$. Two tilings are
\textit{locally isomorphic} if any patch in either tiling also appears in the
other tiling.

The Ammann tilings may be defined, as an alternative to matching rules, by substitution rules, see \cite{durand, TE}.  
Another method of construction appears in \cite{polygon} and is used in Section~\ref{sec:dr} of this paper. 
\vskip 2mm

\begin{figure}[htb]
\centering
\includegraphics[width=5.5cm]{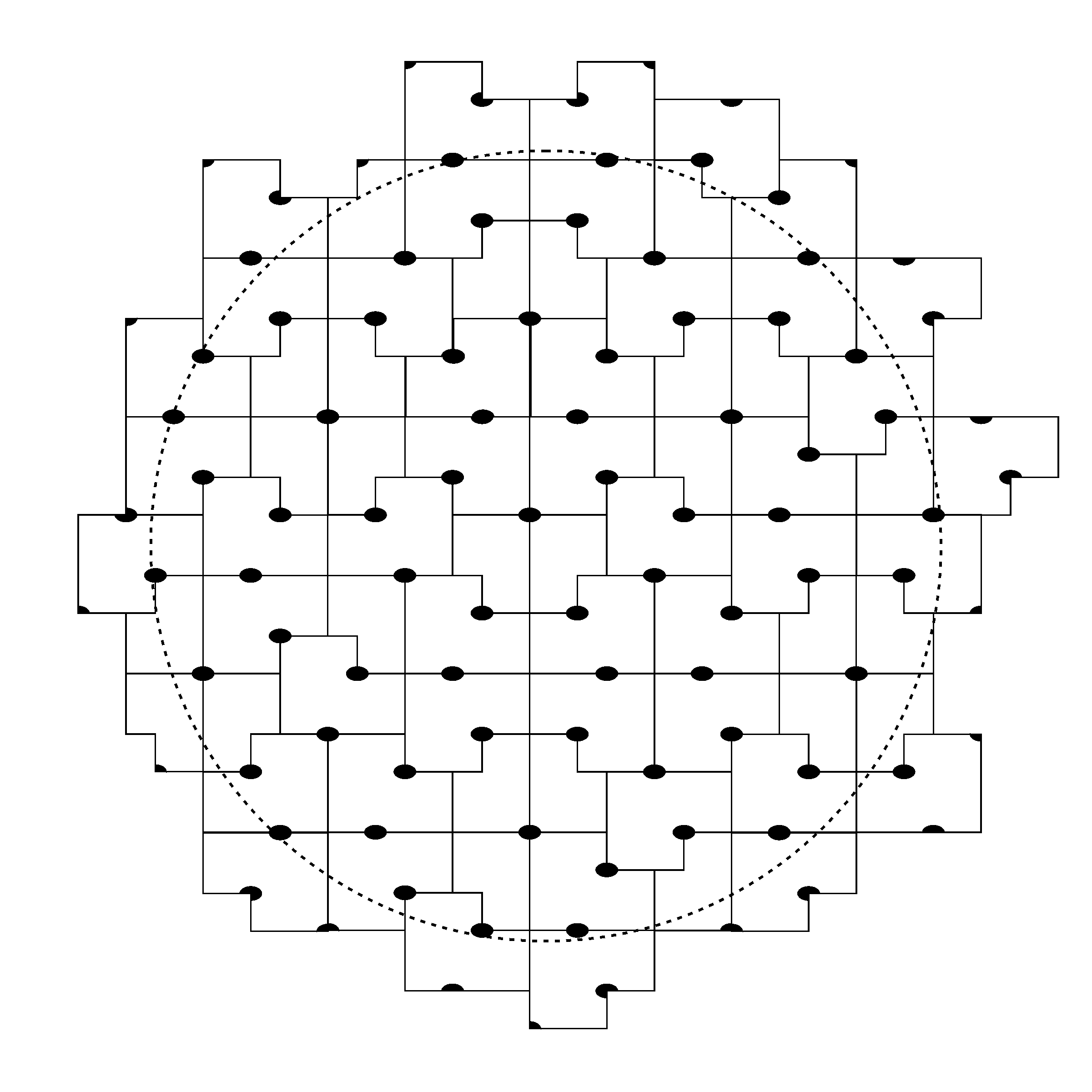}\hskip 5mm
\includegraphics[width=5.5cm]{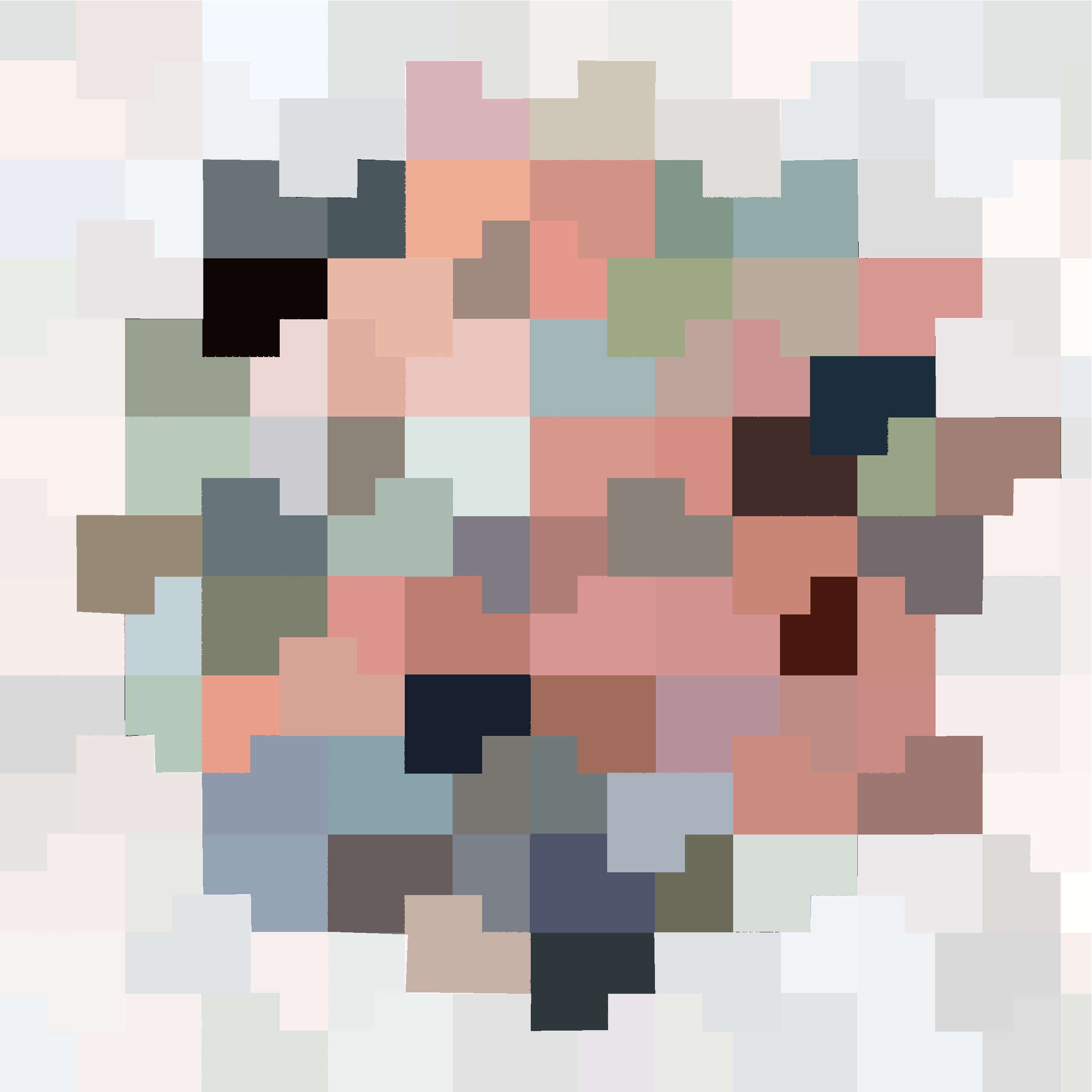}
\caption{Left: a circular patch of an Ammann-A2 tiling, with markings as in Figure
\ref{fig:figA2X}. Right: the same patch, with tiles in various colors, and
part of a possible continuation using faded tiles.}
\label{figa3}
\end{figure}

\section{Amalgamation and the Hierarchy}
\label{sec:A}

Let $T$ be an Ammann tiling. For each small tile $\mathsf{s}$ in $T$, there is
a \textit{partner}, defined as the unique large tile $\mathsf{b}$ in $T$ such
that the union of $\mathsf{s}$ and $\mathsf{b}$, called the
\textit{amalgamation} of $\mathsf{s}$ and $\mathsf{b}$, is a hexagon congruent to $G$.
Any two distinct small tiles have distinct partners. The tiling obtained from
$T$ by amalgamating each small tile with its partner is denoted $A(T)$. The scaled tiling
$a(T):=sA(T)$ is an Ammann tiling called the \textit{amalgamation} of $T$.

\begin{proposition} \label{prop:a}
\label{prop:amal} If $T$ is an Ammann tiling, then so is $a(T)$.
\end{proposition}

\begin{proof} 
 The upper  left and right panels of Figure~\ref{fig:figA2X}
show two different decorations on the small and large tiles, 
call them the $D1$ and the $D2$ decorations, respectively.  

Each small tile in $a(T)$ (left tile in the right panel of Figure~\ref{fig:figA2X}) is a (scaled by $s$)
 large tile in $T$ (right tile in the left panel); 
and each large tile in $a(T)$ (right tile in the right panel) is
the scaled amalgamation of a small and large tile in $T$ (bottom of the left panel).  

As previously noted, the matching rules dictated by the $D1$ and $D2$ 
decorations are equivalent in that they both define the Ammann tilings.  
Therefore if $T$ is an Ammann tiling, then so is $a(T)$.  
\end{proof}

Related comments on matching rules for the Ammann tilings are given in \cite{akiyama, ammann, durand}.
In particular see \cite{akiyama} for a discussion of replacement of markings and
the notion of ghost markings.  Equivalence
of different sets of markings is discussed in \cite{korotin}.

The sequence of tilings $H(T):=\{T,A(T),A^{2}(T),\dots\}$, consisting of
Ammann tilings at larger and larger scales, will be referred to as the
\textit{hierarchy of} $T$. In general, if $t$ is any tile at any level of the
hierarchy, then $t$ is the non-overlapping union of tiles in $T$. Denote this
tiling of $t$ by $\mathcal{T}(t)$. The tiling $\mathcal{T}(t)$ is sometimes
referred to as a \textit{supertile} in the tilings literature; see for
example{\ \cite{frank, grunbaum, goodman}}.  Any large tile $t$ in $A^n(T)$ is 
congruent to $\tau^{n-1}(G)$.  For such a large tile $T$ in $A^n(T)$, denote $\mathcal T(t)$ by $\mathcal T_n$ which, by
Proposition~\ref{prop:levels} below, is well-defined, independent of $t$.  The tiling $\mathcal T_6$ appear in Figure~\ref{fig:A6}.

\begin{figure}[hbt]
\centering
\includegraphics[height=2.7657in,width=2.1802in]{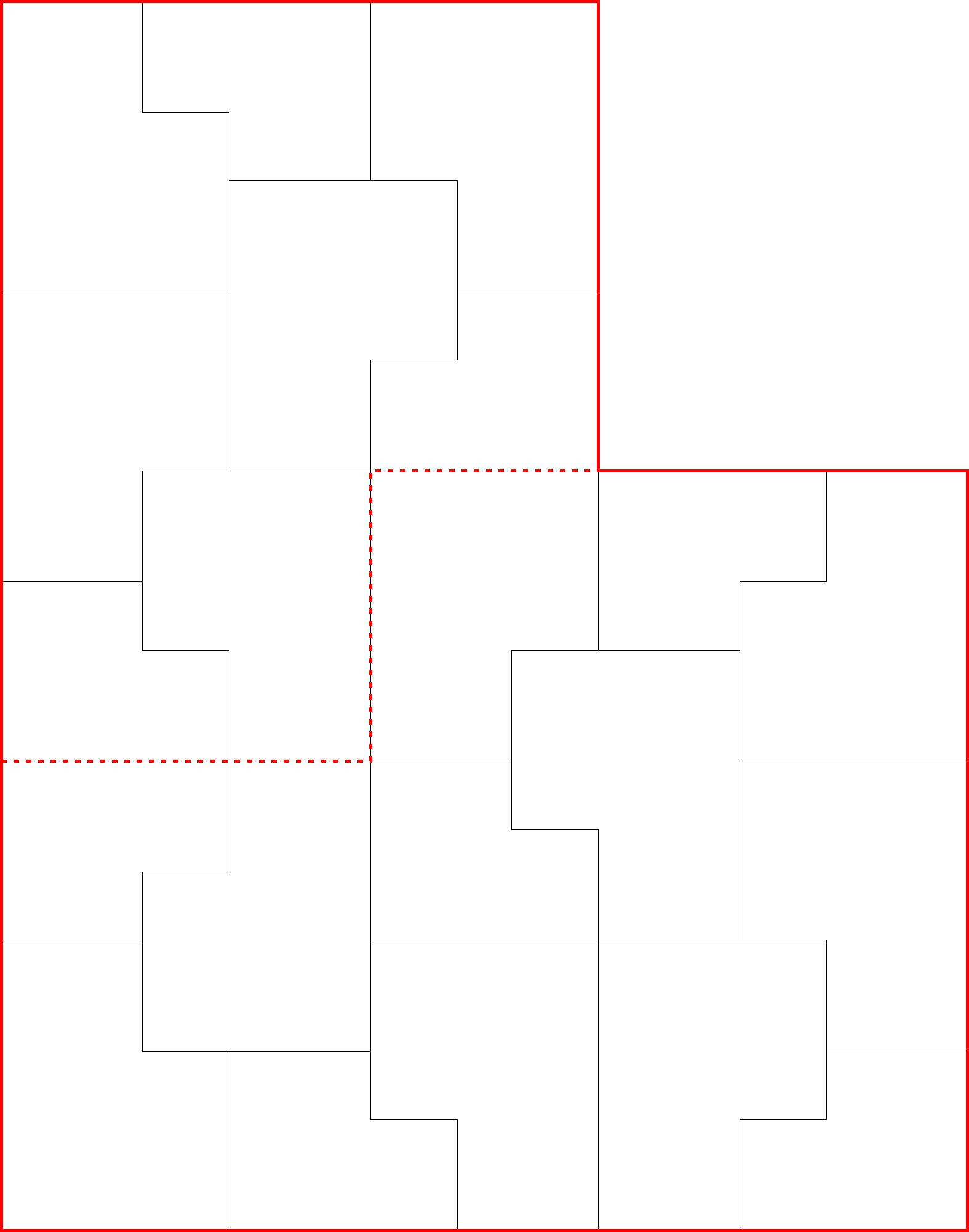}
\caption{The supertile $\mathcal{T}_{6}$.}
\label{fig:A6}
\end{figure}

\begin{proposition}
\label{prop:levels} If $p$ and $q$ are congruent tiles at any
level of the hierarchy, then $\mathcal{T}(p)$ and $\mathcal{T}(q)$ are congruent.
\end{proposition}

\begin{proof}
It is true at level $1$; $A(T)$ is the non-overlapping union of a small and a
large Ammann tile in a unique way. Proceeding by induction on the level,
assume that the statement of the proposition is true at level $n-1$, and
consider congruent tiles $p$ and $q$ of $A^{n}(T)$. Then either (1) $p$ and
$q$ are small tiles of $A^{n}(T)$, in which case they are large tiles of
$A^{n-1}(T)$, or (2) $p$ and $q$ are large tiles of $A^{n}(T)$,
in which case they are the amalgamation of a large and a small tile of
$A^{n-1}(T)$. In case (1) we have that $\mathcal{T}(p)$ and
$\mathcal{T}(q)$ are congruent by the induction hypothesis. In case (2) the
decomposition of $p$ into a non-overlapping union of a large and small tile of
$A^{n-1}(T)$ is the same as for $q$. Moreover, by the induction hypothesis,
the decompositions of large and small tiles of $A^{n-1}(T)$ is unique.
Therefore we again have that $\mathcal{T}(p)$ and $\mathcal{T}%
(q)$ are congruent.
\end{proof}

Lemma ~\ref{lem:map} below is needed to prove a main result in Section~\ref{sec:ca}.  
We first introduce a useful combinatorial notion. Define
a \textit{planar map} $M$ to be a 2-cell embedding of a locally finite simple
graph $\Gamma$ in the plane. By \textit{locally finite} is meant that the
degree of each vertex is finite, and by \textit{simple} is meant no loops or multiple edges.   
The \textit{faces} of $M$ are the closures of
the connected components of ${\mathbb{R}}^{2}\setminus\Gamma$. By
$2$-\textit{cell embedding} is meant that each face of $M$ is homeomorphic to
a closed disk. Two planar maps $M$ and $M^{\prime}$ with underlying graphs $\Gamma$ and $\Gamma '$, respectively,
 are {\it isomorphic} if there is
a graph isomorphism taking $\Gamma$ to $\Gamma^{\prime}$ that preserves faces.

A tiling $T$ can be considered a planar map. A \textit{vertex} of $T$ is
the intersection point of three or more distinct tiles of $T$, assuming that
the intersection is not empty, and an \textit{edge} of $T$ is the
intersection, if not empty or a single point, of two distinct tiles of $T$.
The \textit{faces} of $T$ are the tiles of $T$. 

From an Ammann tiling $T$, a tiling $\widehat T$ is constructed as follows.
 Note that the degree (number of incident edges)
of any vertex of an Ammann tiling $T$ is either $3$ or $4$. Color red each edge of 
$T$ (and its two incident vertices) that
joins two vertices of degree $4$. For any red vertex lying on only one red
edge, remove the color red from that vertex and remove the incident edge. Let
$\widehat{T}$ be the tiling induced by the red edges and vertices, i.e., by
removing all edges and vertices not colored red. For any face $f$ of
$\widehat{T}$, it's red boundary $\partial f$ together with all enclosed tiles
of $T$ is a finite tiling which is denoted $T(f)$.

\begin{lemma}
\label{lem:map} With notation as above, if $T$ is an Ammann tiling, then $\widehat{T}$ has the
following properties.

\begin{enumerate}
\item If $f$ is a tile of $\widehat{T}$, then $T(f) = \mathcal{T}_{4}$ or
$T(f)= \mathcal{T}_{5}$.

\item $\widehat{T} = A^{5}(T)$.
\end{enumerate}
\end{lemma}

\begin{proof}
In $\mathcal T_n$, any
edge adjacent to exactly one tile is called a \textit{boundary edge}. A
\textit{boundary vertex} is a vertex that lies on a boundary edge. In
$\mathcal{T}_{n}$, color edges and vertices red by the same rule as used to
color $T$, but, in addition, color red any edge (and incident vertex) joining
a vertex of degree $4$ to a boundary vertex and color all boundary edges (and
incident vertices) red. As previously, for any red vertex lying on only one
red edge, remove the color red from that vertex and remove the incident edge. Denote
by $\widehat{\mathcal{T}_{n}}$ the tiling induced by only the red edges and
vertices. Figure~\ref{fig:A6} shows the tiling $\mathcal{T}_{6} $ with red
edges designated and the corresponding tiling $\widehat{\mathcal{T}_{6}}$.
Note that, for any red edge $e$ in $\mathcal{T}_{6}$ not on the boundary, one
incident tile is the reflection of the other incident tile in the line that
extends $e$. Also note that $\widehat{\mathcal{T}_{6}}$ has two faces, say $f$
and $f^{\prime}$ such that $T(f)=\mathcal{T}_{4}$ and $T(f^{\prime
})=\mathcal{T}_{5}$. We claim that the same is true for all $\mathcal{T}%
_{n},n\geq6$, i.e.,  (1) for any red edge $e$ of $\mathcal{T}_{n}$ not on the
boundary, one incident tile is the reflection of the other in the line that
extends $e$ (and this implies that any red vertex of $\mathcal{T}_{n}$, not on
the boundary, has degree $4$), and (2) for any tile $f$ of
$\widehat{\mathcal{T}_{n}}$ it is the case that $T(f)=\mathcal{T}_{4}$ or
$T(f)=\mathcal{T}_{5}$.

The claim will be proved by induction on $n$. Assume that assertions (1) and
(2) above are true for $\mathcal{T}_{n}$ and $\widehat{\mathcal{T}_{n}}$.
Obtain $\mathcal{T}_{n+1}$ by first subdividing each large tile in
$\mathcal{T}_{n} $ into one large and one small tile, then enlarging the
resulting tiling by a factor of $\tau$. Now color the edges of $\mathcal{T}%
_{n+1}$ as previously prescribed. If $e$ is a red edge in $\mathcal{T}_{n}$
and $f$ and $f^{\prime}$ are the adjacent tiles, then either they are both
small tiles in $\mathcal{T}_{n}$, and hence not subdivided, or they are both
large tiles in $\mathcal{T}_{n} $, and hence subdivided in exactly the same
way, so that the reflection property remains true in $\mathcal{T}_{n+1}$.
Therefore, each red edge in $\mathcal{T}_{n}$ induces either one or two red
edges in $\mathcal{T}_{n+1}$. Let $E$ be this set of induced red edges and let
$\mathcal{T}^{\prime}_{n+1} $ be $\mathcal{T}_{n+1}$, but with only $E$ and
the boundary edges (and incident vertices) colored red. Since each tile of
$\widehat{\mathcal{T}_{n}}$ is either $\mathcal{T}_{4}$ or $\mathcal{T}_{5}$,
each tile in $\mathcal{T}^{\prime}_{n+1}$ is either $\mathcal{T}_{5}$ or
$\mathcal{T}_{6}$. But when additional red faces are added to $\mathcal{T}%
^{\prime}_{n+1}$ to obtain $\mathcal{T}_{n+1}$, no additional red edges are
added to each $\mathcal{T}_{5}$ because $\mathcal{T}_{5}$ has no degree $4$
internal (not on the boundary) vertices, and the additional red edges added to
each $\mathcal{T}_{6}$, as shown in Figure~\ref{fig:A6}, divides each
$\mathcal{T}_{6}$ into one $\mathcal{T}_{4}$ and one $\mathcal{T}_{5}$. Our
claim has now been proved.

To extend the result from $\widehat {\mathcal T_n}$ to $\widehat T$,
 let $f$ be any tile of $\widehat{T}$. If $f$ lies in the interior (no edge
of $f$ on the boundary) of a tile in $A^{n}(T)$ for some $n$, then by the
paragraphs above, $T(f)$ is either $\mathcal{T}_{4}$ or $\mathcal{T}_{5}$. So
assume that there is an edge $e$ of $f$ that lies the on the boundary of some 
tile $t_{n}$ of $A^{n}(T)$ for all $n$ sufficiently large, say $n\geq n_{0}$.
Denote by $T_{n}$ the subset of the tiling $T$ that lies in $t_{n}$. Then
$T_{n_{0}}\subset T_{n_{0}+1}\subset T_{n_{0}+2}\subset\cdots$. The nested
union $\bigcup_{n\geq n_{0}}T_{n}$ is an Ammann tiling of a proper subset of
the plane. But it is known \cite{durand} that this occurs only if it is a
tiling of a half-space bounded by the line $L$ that extends $e$ or a quadrant
of the plane bounded by two perpendicular rays $L_{1}$ and $L_{2}$. Moreover,
if $T$ is an Ammann tiling of the half plane, then the only extension to an
Ammann tiling of the entire plane is obtained by reflecting $T$ in the line
$L$ of $T$. Similarly for a tiling obtained from a tiling $T$ of a quadrant
by reflecting in the two perpendicular border lines $L_{1}$ and $L_{2}$. In
the half plane case, the edges and vertices of $T$ on $L$ will be colored red,
and in the quadrant case, the edges and vertices of $T$ on $L_{1}$ and $L_{2}$
will be colored red. So again, $T(f)$ is either $\mathcal{T}_{4}$ or
$\mathcal{T}_{5}$, and statement (1) of Lemma~\ref{lem:map} is proved.

Concerning statement (2), the tiles in both $A^{5}(T)$ and $\widehat{T}$ are
congruent copies of $\tau^{5}(G)$ and $\tau^{4}(G)$. Let $t\in\widehat{T}$. If
$\mathsf{s}$ is an small tile in $A^{j}(T),\,j<5,$ that is contained in $t$,
and $\mathsf{b}$ is its partner, then the shapes dictate that $\mathsf{b} $ is
also contained in $t$. Therefore $\widehat{T}=A^{5}(T)$.
\end{proof}

\section{Bounded Distortion Homeomorphism}
\label{sec:bdh}

\begin{definition}
A homeomorphism $h$ of the plane has {\bf bounded distortion} if there is a
constant $C$ such that
\[
\label{eq:bd}|x - h(x)| \leq C
\]
for all $x\in{\mathbb{R}}^{2}$. In this case, $h$ is said to have
{\bf bounded distortion} $C$.
\end{definition}

\begin{lemma} \label{lem:h1}
Let $h$ be a homeomorphism of the plane with bounded
distortion $C$. If $\phi_r$ is a similarity transformation with scaling ratio
$r$, then
\[
|x- \phi^{-1} \circ h\circ\phi(x)| \leq \frac1r \, C\]
for all $x \in \R^2$.
\end{lemma} 

\begin{proof}
Let $x\in \R^2, \, y=\phi_r(x), \, z= h(y), \, w= \phi_r^{-1}(z)$. Then
\[\begin{aligned}
|x- \phi_r^{-1} \circ h\circ\phi_r(x)| &= |x -w| = |\phi_r^{-1} (y) - \phi_r^{-1}(z)| =
\frac1r \, |y-z| \\ & = \frac1r \, |y-h(y)|\leq \frac1r \, C
\end{aligned}\]
for all $x\in \R^2$. 
\end{proof}

Let $\mathbb X$ denote the set of tilings as defined in Section~\ref{sec:a2}.  For a tiling $T \in \mathbb X$,
  let $\partial T$ denote the union of the boundaries of the tiles of $T$.
 Define a distance function on $\mathbb X$ as follows.  

\[d(T,T^{\prime}) = d_{H}(\partial T, \partial T^{\prime}).\] 
where $d_H$ is the Hausdorff distance.  
Intuitively, the $d$-distance between tilings is small if the tilings are almost the same over the entire plane.  
The next corollary follows 
immediately from Lemma~\ref{lem:h1}.

\begin{cor}  \label{cor:d1}  If $h$ a  homeomorphism of bounded distortion $C$ and 
$\phi_r$ is a similarity transformation with scaling ratio $r$, then 
\[d \big  (\phi_{r}^{-1} \circ h\circ\phi_{r})(T), T \big)  \leq \frac1r \, C\] 
for all tilings $T \in \mathbb X$.    
\end{cor}  

If $T$ is an Ammann tiling and $h$ is a homeomorphism of the plane, then $h$
acts on $T$. Let $T^{\prime}=h(T)$ be the image of an Ammann tiling $T$ under
a homeomorphism $h$. The small and large tiles in $T^{\prime}$ are defined to
be the images of the small and large tiles in $T$.  If $\mathsf  s '$ is a small
tile in $T'$ corresponding to a small tile $\mathsf s$ in $T$, then the partner of $\mathsf s '$
 is the image under $h$ of the partner of  $\mathsf s$.  Therefore the amalgamation
$a(T^{\prime})$ and the hierarchy $H(T^{\prime})$ can be defined for $T'$ exactly as
they are for $T$.  
 
\begin{theorem}  \label{thm:m}  
Let $T$ be an Ammann tiling, $h$ a bounded distortion
homeomorphism of the plane, and $T^{\prime}=h(T)$. Then
\[
\lim_{k\rightarrow\infty}d\big (a^{k}(T^{\prime}),a^{k}(T)\big )=0.
\]
\end{theorem}

\begin{proof}

Let $H(T)=\{T=T_{0},T_{1},T_{2},\dots \}$ be the hierarchy of $T$ and
$H(T^{\prime})=\{T^{\prime}=T_{0}^{\prime},T_{1}^{\prime},T_{2}^{\prime},\dots\}$ the hierarchy of $T^{\prime}$. 
Then

\[
\begin{aligned} a^k(T) &= s^k T_k \\ a^k(T') &= s^k T'_k= s^k h (T_k) = s^k h (s^{-k} a^k(T) ) = (s^k\circ h \circ s^{-k}) \big (a^k(T) \big ) \\ &= (\tau^{-k}\circ h \circ \tau^{k}) \big (a^k(T) \big ) \end{aligned}
\]
 where we recall that $\tau=s^{-1}$.  
 By Corollary~\ref{cor:d1} we have
\[
d\big (a^{k}(T^{\prime}),a^{k}(T)\big )= d\big ((\tau^{-k}\circ h\circ\tau^{k})\big (a^{k}
(T)),a^{k}(T)\big ) \leq \frac{C}{\tau^k}.
\]
From this it follows that 
\[
\lim_{k\rightarrow\infty}d\big (a^{k}(T^{\prime}),a^{k}(T)\big )=0.  \qedhere
\]
\end{proof} 

\begin{remark} \label{rem:converge}
Theorem~\ref{thm:m} does not state that $a^k(T')$ becomes arbitrary close to $T$ as $k\rightarrow \infty$.  It does imply, however,
that an arbitrarily large finite patch of the tiling $a^k(T')$ becomes arbitrarily close to a patch $P_k$ of $T$ as  $k\rightarrow \infty$.  This
follows from Proposition~\ref{prop:a} and the fact that any two Ammann tilings are locally isomorphic.  
\end{remark}

\section{Combinatorial Amalgamation} 
\label{sec:ca}

Let $\mathbb{T}$ be the set of all images of
Ammann tilings of the plane under homeomorphisms of the plane. 
Essential to Theorem~\ref{thm:m} is being able to apply the amalgamation operator
$a:\mathbb{T}\rightarrow\mathbb{T}$.  Let $T$ be an Ammann tiling and
$T^{\prime}=h(T)$, where $h$ is a homeomorphism of the
plane. In this section it is proved that $a(T^{\prime})$ can be determined
combinatorially, without knowledge of the homeomorphism $h$.  

\begin{theorem} \label{thm:ed}   
Let $T$ be an Ammann tiling, $h$ a homeomorphism of the plane
and $T^{\prime}= h(T)$. Then Algorithm A below computes
$a(T^{\prime})$, without knowledge of $h $.
\end{theorem}

The proof of Theorem~\ref{thm:ed} appears after a few simple lemmas.  
With terminology from Section~\ref{sec:A}, the following lemma is clear.

\begin{lemma}
\label{lem:same} If $T$ is a tiling of the plane and $h$ is a homeomorphism of
the plane, then as planar maps $T$ and $h(T)$ are isomorphic.
\end{lemma}

Two planar maps $M_1$ (left) and $M_2$ (right) are shown in Figure~\ref{fig:pg} (ignoring the dots and arrows). 
Recalling the definition of $\mathcal T_n$ from Section~\ref{sec:A}, the 
following two lemmas are apparent by inspection.  

\begin{lemma}
\label{lem:pg} The two planar maps $M_1$ and $M_2$ in Figure~\ref{fig:pg} are isomorphic to
$\mathcal{T}_{4}$ and $\mathcal{T}_{5}$, respectively, considered as planar maps.
\end{lemma}

\begin{lemma}
\label{lem:aut} Each of the two planar maps in Figure~\ref{fig:pg}, ignoring
the arrows, has trivial automorphism group.
\end{lemma}

\begin{figure}[tbh]
\centering\includegraphics[width=4in]{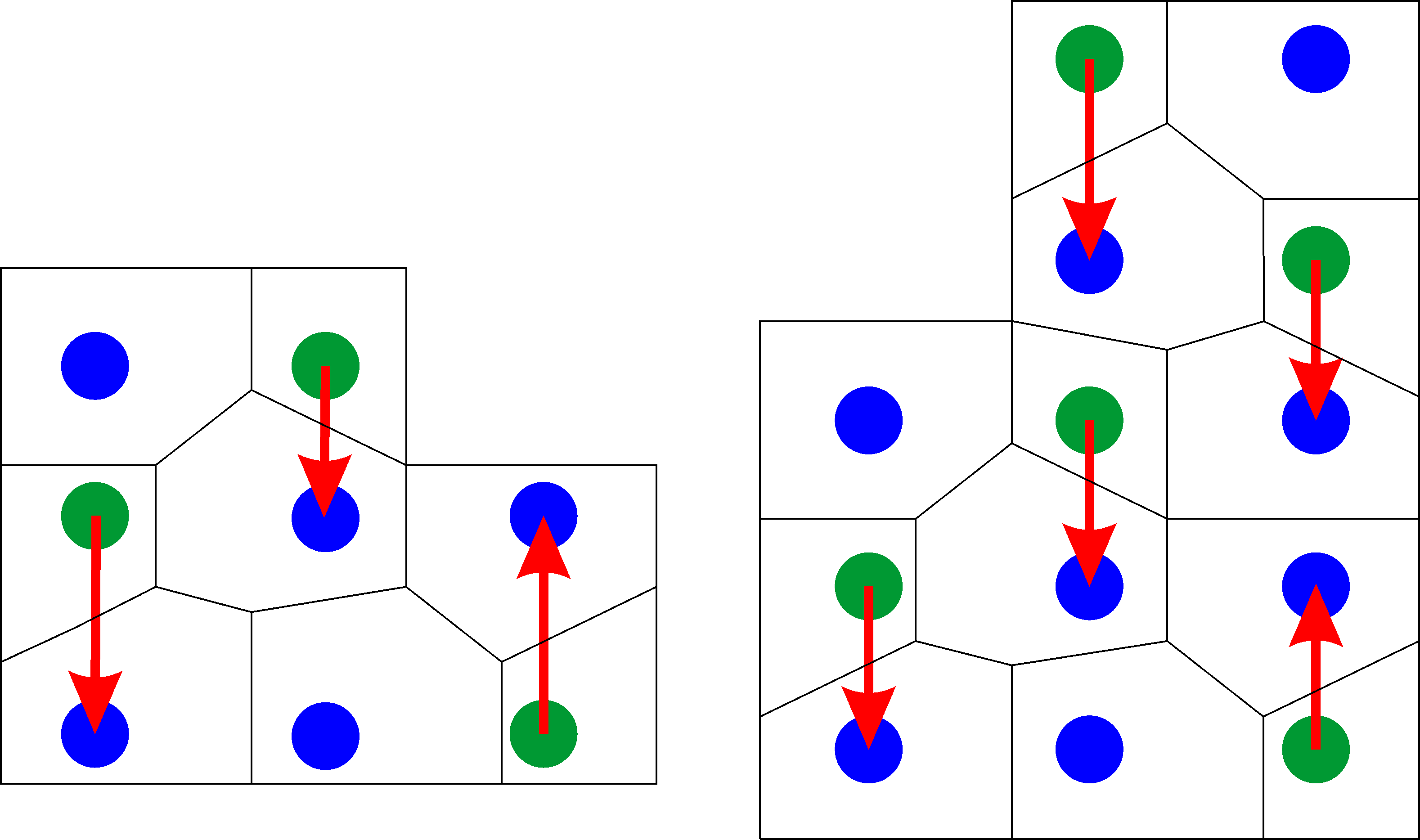}
\caption{Rule for locating the partner of a small tile.}
\vskip -1mm
\label{fig:pg}
\end{figure}

Each arrow in Figure~\ref{fig:pg} points from a face bounded by a
$4$-cycle to an adjacent face. A consequence of Lemma~\ref{lem:aut} is that,
if $M$ is any planar map isomorphic to $M_1$ or $M_2$ in
Figure~\ref{fig:pg}, then arrows can be uniquely assigned to $M$ such that
there is an arrow from face $f_{1}$ to $f_{2}$ in $M$ if and only if there
is an arrow from $\phi(f_{1})$ to $\phi(f_{2})$ in  $M_1$ or $M_2$, where $\phi$ is
the isomorphism. 
\vskip 3mm

\noindent\textbf{Algorithm A} 
\vskip 2mm 
\noindent Input: A tiling $T'$ that is the image of 
an Amman tiling, distorted by an 

\hskip 2mm unknown homeomorphism. 
\vskip 2mm 

\noindent Output: The amalgamated tiling $a(T^{\prime})$. 
\vskip 2mm

Perform the following steps to obtain the tiling $a(T^{\prime})$ from $T$.
\vskip 2mm

\begin{enumerate}
\item Color red each edge (and its two incident vertices) of $T^{\prime}$ that joins 
two vertices of degree $4$.  But, for any red vertex lying on only one red
edge, remove the color red from that vertex and remove the incident edge. 
\vskip 2mm

\item Let $\widehat{T^{\prime}}$ be the tiling induced by just the red edges and
vertices of $T^{\prime}$,  i.e., by removing all edges and vertices not colored red.
\vskip 2mm

\item For every tile $f \in \widehat{T^{\prime}}$, let $T^{\prime}(f)$ be the set
of tiles of $T^{\prime}$ contained in $t$. Viewed as a planar map,
each $T^{\prime}(f)$ is isomorphic to one of the planar maps in
Figure~\ref{fig:pg}.  (This is shown to be the case in the proof below of Theorem~\ref{thm:ed}.) For
every $f \in \widehat{T^{\prime}}$ do: 
\vskip 2mm
\begin{itemize}
\item For every tile $\mathsf{s}$ of $T^{\prime}(f)$ that is bounded by a $4$-cycle, do:
\begin{itemize}
\item Locate the partner $\mathsf{b}$ of $\mathsf{s}$ using the relevant arrow in Figure~\ref{fig:pg}.
\item Replace the two tiles $\mathsf{s}$ and $\mathsf{b}$ in $T^{\prime}$ by their union to form a single tile in $A(T^{\prime})$,
and hence in $a(T')$ after scaling by $s = 1/\tau$.  
\end{itemize}
\end{itemize}
\end{enumerate}
\vskip 3mm

\begin{proof}
[Proof of Theorem~\ref{thm:ed}]  The notation in this proof follows the notation in Section~\ref{sec:A} just
prior to Lemma~\ref{lem:map}.
Let $\mathsf{s}
^{\prime}$ be any small tile in $T^{\prime}$. It must be shown that
Algorithm A pairs $\mathsf{s}^{\prime}$ with its large partner tile in
$T^{\prime}$.  The tile $\mathsf{s}$ is contained in a unique tile $f\in
A^{5}(T) = \widehat T$, the last equality by statement (2) of Lemma~\ref{lem:map}.
Therefore $\mathsf{s}^{\prime}:=h(\mathsf{s})$ is contained in $f^{\prime} := h(f) \in A^5(T')$.
Since $f\in\widehat{T}$, we know that $f^{\prime}\in
h(\widehat{T})$.

By Lemma~\ref{lem:same}, $T$ and $h(T)$, as well as $\widehat{T}$ and
$h(\widehat{T})$, are isomorphic as planar maps.
Therefore, by statement (1) of  Lemma~\ref{lem:map}, if $\mathcal{T}(f^{\prime})$ denotes the set of all tiles in $T^{\prime}$
contained in $f^{\prime}$, then $\mathcal{T}(f^{\prime})$ is isomorphic to
$\mathcal{T}_{4}$ or $\mathcal{T}_{5}$. By Lemma~\ref{lem:pg}, we know that
$\mathcal{T}(f^{\prime})$ is isomorphic to one of the planar maps in
Figure~\ref{fig:pg}.

 Note that the boundary of any face in $T^{\prime}$ is a $4
$-cycle if the corresponding tile is small and a $5$ or $6$-cycle if the
corresponding tile is large. Therefore, by Lemma~\ref{lem:aut} and the remarks
following it, the partner of $\mathsf{s}^{\prime}\in f^{\prime}$ is uniquely
determined by the arrows in Figure~\ref{fig:pg}. This completes the proof of
Theorem~\ref{thm:ed}.
\end{proof}

Algorithm A acts on the tiling $T'$ of the entire plane and hence does not run in finite time.  
From a practical point of view, it may be asked whether it is possible to efficiently
 compute the amalgamation of the distorted tiling $T'$ on a finite subset of the plane, for instance on a patch
containing the disk $D_R$ of radius $R$ centered at the origin.  We tweak Algorithm A to obtain a 
combinatorial Algorithm B that 
 applies to a finite patch of a tiling.  Algorithm B computes $a(T^{\prime})$ on $D_R$ and runs in time quadratic in $R$.
This is the content of Theorem~\ref{thm:ed2} below.  Figure \ref{bee4newdistort} illustrates Algorithm A above and Algorithm B below, showing the red edges of a distorted Ammann tiling before and after the red edges that do not belong to cycles are removed.

The notation $T'|R$ and $a(T^{\prime}|R)$ in Theorem~\ref{thm:ed2} are defined as follows.  
If $T'$ is a tiling, $T'|R$ the set of tiles of $T'$ with
non-empty intersection with $D_{R}$. If $T'$ is an Ammann tiling or a
homeomorphic image of an Ammann tiling, the \textit{amalgamation} $a(T'|R)$ is
obtained by replacing every small tile $t$ in $T|R$ by the union of $t$ and
its partner - even if the partner of $t$ has empty intersection with $D_{R}$.

\begin{lemma}
\label{lem:bounded} Let $R$ be a positive real number. If $T$ is an Ammann
tiling and $h$ is a homeomorphism of bounded distortion $C$, then

\begin{enumerate}
\item $h(A^{5}(T))|R \subset D_{R+2C+\tau^{6}}$,
\item the number of tiles of $h(T)$ contained in the disk $D_{R}$ is less than \linebreak
$\tau \pi (\frac{15}{4}-\sqrt{5})(R+C)^{2}.$
\end{enumerate}
\end{lemma}

\begin{proof}
Concerning statement (1), if $x$ is any point in a tile $t' \in h(A^{5}(T))|R$, and $y$ is any point in $t^{\prime}\cap D_{R}$,
then by the triangle inequality and the fact that $h$ is of bounded distortion
$C$, we have
\[
\begin{aligned} |x| & \leq |y| + |x-y| \leq R + diam(t') = R + diam(h(t)) \leq R + diam(t) + 2C \\&\leq R + diam (\tau^5 G) + 2C = R+\tau^6 +2C, \end{aligned}
\]
where $diam$ denotes the diameter of the set.  

Concening statement (2), because $h$ is of bounded distortion $C$, we have
$D_{R} \subseteq h(D_{R+C})$. Each small tile in an Ammann tiling is congruent to $s^2 G$.  Therefore there are less than 
$area(D_{R+C})/area(s^2 G) = \tau \pi (\frac{15}{4}-\sqrt{5}) (R+C)^{2}$ Ammann tiles in $D_{R+C}$.
\end{proof}
\vskip 2mm

\begin{figure}[htb]
\centering
\includegraphics[width=5.2cm]{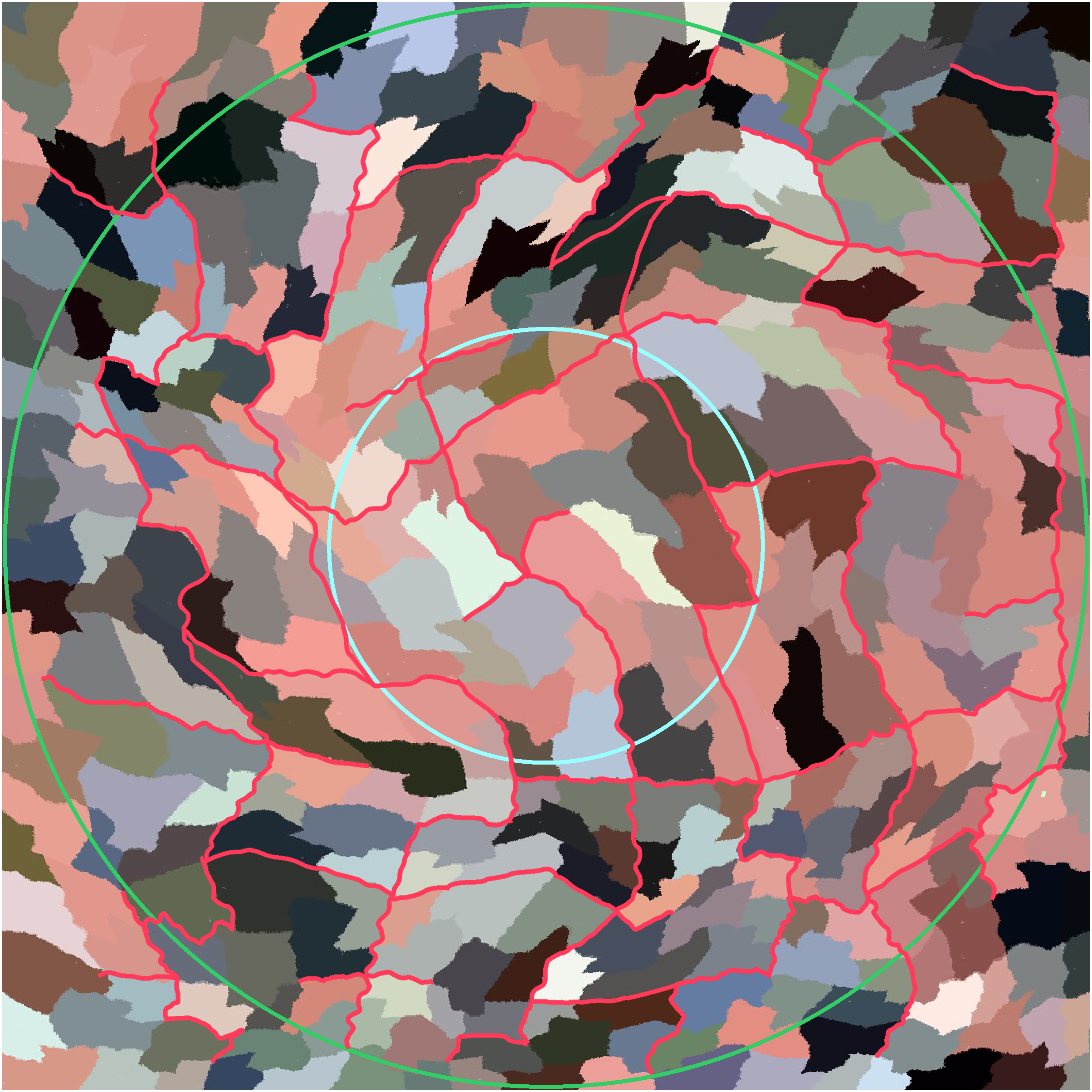}\hskip 10mm
\includegraphics[width=5.2cm]{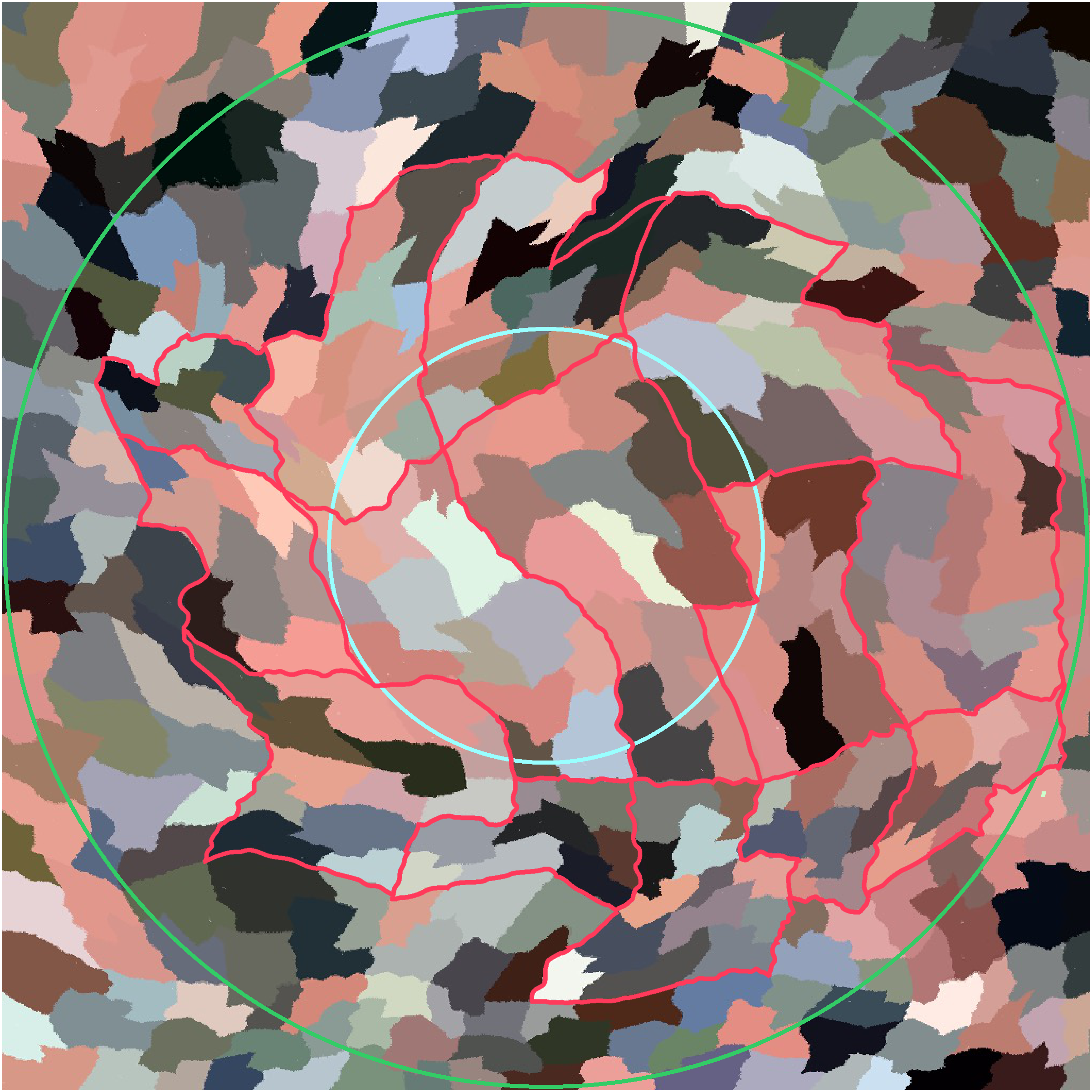}
\vskip 10mm
\includegraphics[width=5.2cm]{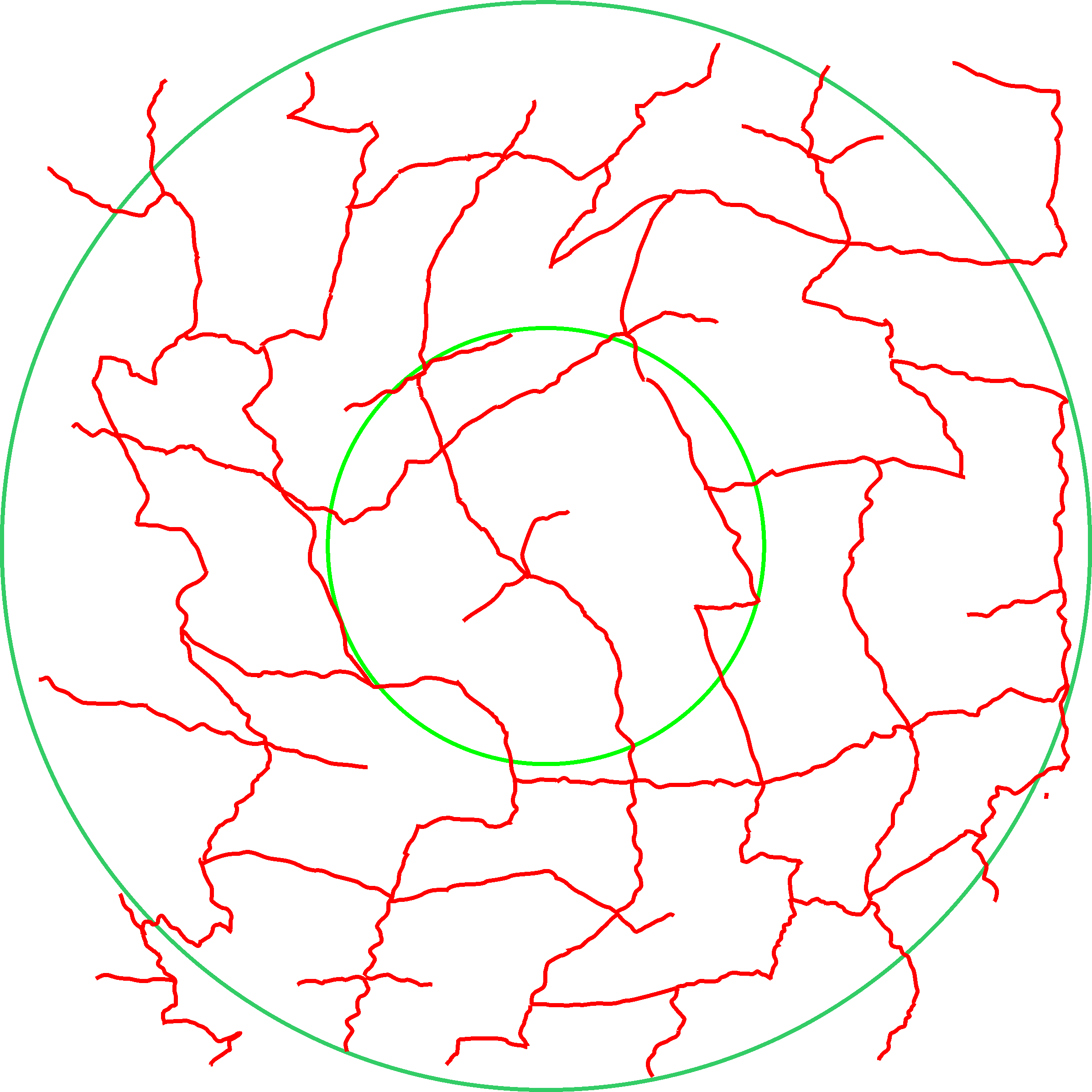}\hskip 10mm
\includegraphics[width=5.2cm]{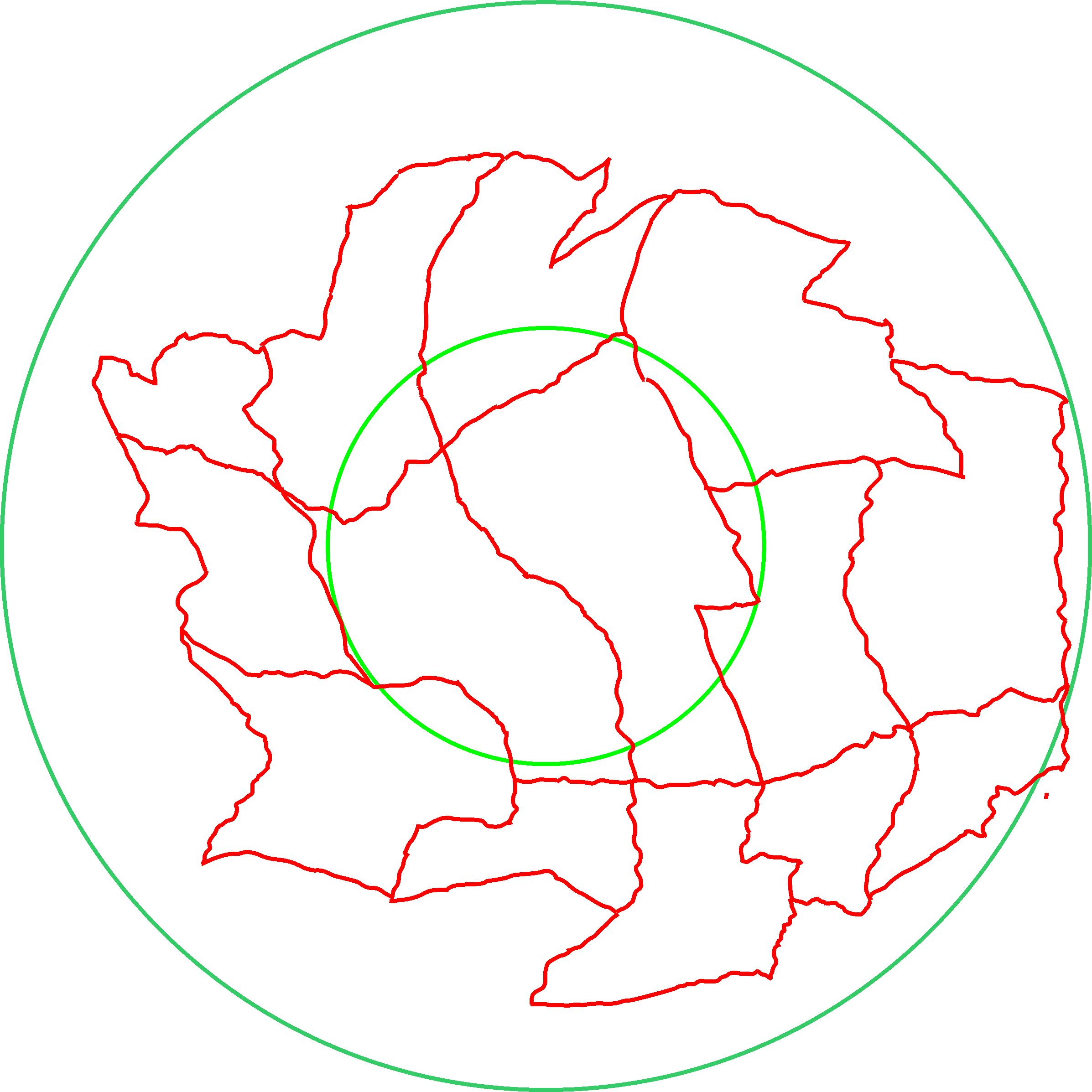}
\vskip 3mm
\caption{Left: a distorted Amman tiling
with edges that connect vertices of degree 4 marked in red. Right: the same
tiling after removal of red edges that do not lie on red cycles.}
\label{bee4newdistort}
\end{figure}
\vskip 2mm

\noindent\textbf{Algorithm B} 
\vskip 2mm 
\noindent Input: A tiling $T'$ that is the image of an Amman tiling, distorted by an 

\hskip 2mm unknown homeomorphism of bounded distortion $C$, and a positive real 

\hskip 2mm number $R$.
\vskip 2mm 

\noindent  Output: A subset
$T^{\prime\prime}$ of the amalgamated tiling $a(T^{\prime})$ that contains
$a(T^{\prime}|R)$. \vskip 2mm
\vskip 2mm

\begin{enumerate}
\item Set $\widehat{R} = R+2C+\tau^{6}$ \vskip 2mm \hskip -16mm
 Perform the following steps to obtain the tiling $T^{\prime\prime}$ from the tiling $T^{\prime}|\widehat{R}$: 
\vskip 2mm

\item Color some edges and vertices of $T^{\prime}|\widehat{R}$ red according
to the following rules:

\begin{itemize}
\item Color red each edge (and its two incident vertices) of $T^{\prime
}|\widehat{R}$ that joins two vertices of degree $4$.

\item Remove the color red from any red vertex or edge not lying on a red
cycle in $T^{\prime}|\widehat{R}$.
\end{itemize}

\vskip 2mm

\item Let $\widehat{T^{\prime}}$ be the tiling induced by the red edges and
vertices of $T^{\prime}|\widehat{R}$. \vskip 2mm

\item For every tile $t \in \widehat{T^{\prime}}$, let $T^{\prime}(t)$ be the set
of tiles of $T^{\prime}|\widehat{R}$ contained in $t$. Viewed as planar map,
each $T^{\prime}(t)$ is isomorphic to one of the planar maps in
Figure~\ref{fig:pg}. For
every $t \in \widehat{T^{\prime}}$ do: \vskip 2mm

\begin{itemize}
\item For every tile  $\mathsf{s}$ of $T^{\prime}(t)$ that is bounded by a $4$-cycle, do:

\begin{itemize}
\item Locate the partner  $\mathsf{b}$ of  $\mathsf{s}$ using the the relevant arrow in
Figure~\ref{fig:pg}.

\item Replace the two tiles  $\mathsf{s}$ and  $\mathsf{b}$ in $T^{\prime}|\widehat{R}$ by their
union to form a single tile in $T^{\prime\prime}$.
\end{itemize}
\end{itemize}
\end{enumerate}
\vskip 5mm

\begin{theorem}
\label{thm:ed2} Let $T$ be an Ammann tiling, $h$ a homeomorphism of the plane
of bounded distortion, and $T^{\prime}= h(T)$. Then Algorithm B above computes
$a(T^{\prime}|R)$ in time quadratic in $R$, without knowledge of $h $.
\end{theorem}

\begin{proof}
In addition to what was shown in the proof of Theorem~\ref{thm:ed}, it must be
verified that the finite tiling $\widehat{T^{\prime}}$ constructed in Algorithm B is 
contained in $D_{\widehat{R}}$, which follows from statement (2) of Lemma~\ref{lem:map} and
startement (1) of Lemma~\ref{lem:bounded}.

Concerning the running time of Algorithm B, it is clear that the number of
steps is $O(n)$, where $n$ is the sum of the number of tiles in $T^{\prime
}|\widehat{R}$. By statement (2) of Lemma~\ref{lem:bounded}, the number of tiles in $T^{\prime
}|\widehat{R}$ is less than
\[
\tau \pi (\frac{15}{4}-\sqrt{5}) (\widehat{R}+C)^{2} = 
\tau \pi (\frac{15}{4}-\sqrt{5}) (R+3C+\tau^{6})^{2} = O(R^{2}). \qedhere
\]
\end{proof}

\section{Direct Retrieval of an Ammann tiling from a Distorted Image}  \label{sec:dr}

Let $T$ be an Ammann tiling and $T^{\prime}=h(T)$, where $h$ is a
homeomorphism of the plane. In this section we extract from $T^{\prime}$,
without using $h$, an infinite binary string $\theta= \theta_{1} \, \theta
_{2}\, \, \theta_{3 }\cdots$. A tiling $T(\theta)$ will be combinatorially
constructed from $\theta$ such that $T=T(\theta)$.

\begin{definition} \label{def:code}
Let $\mathbb{S}$ be the set of all pairs 
$(T^{\prime},t^{\prime})$, where $T' = h(T)$ and $t' = h(t)$
for some Ammann tiling $T$, large tile $t\in T$, and homeomorphism $h$ of the plane.  
 Let $\{1,2\}^{\infty}=\{c_{1}\,c_{2}
\,c_{3}\cdots:c_{i}\in\{1,2\}\; \text{for all}\, i\}$, i.e., the set of infinite binary stings. Also
let $\{1,2\}^{\ast}$ denote the set of finite binary strings. To define a map
$g:\mathbb{S}\rightarrow\mathbb{S}$ and a function $c:\mathbb{S}
\rightarrow\{1,2\}^{\infty}$, called the \textit{code} of $(T^{\prime
},t^{\prime})$, let
$(T^{\prime},t^{\prime})\in\mathbb{T}$. Define $(\widetilde{T},\widetilde{t})=g(T^{\prime},t^{\prime})$  
 as follows. There are two cases. \vskip2mm

Case 1. If $t^{\prime}$ has a partner in $T^{\prime}$, then let $\widetilde{T} = A(T')$, the amalgamation of $T^{\prime}$,
 and let $\widetilde{t}$ be union of $t'$ and its partner, a large tile in $\widetilde{T}$. Define
$g(T^{\prime},t^{\prime}) = (\widetilde{T},\widetilde{t})$ and $\widetilde{c}(T^{\prime},t^{\prime}) = 1$. 
\vskip 2mm

Case 2. If $t^{\prime}$ has no partner in $T'$, then let $\widetilde{T} = A^2(T')$, the second amalgamation of $T^{\prime}$.
 Then $t^{\prime}$ is a small tile in $A(T')$  that has a partner.  In $\widetilde{T}$ 
 let $\widetilde{t}$ be the union of $t^{\prime}$ and
its partner in $A(T')$.  Let
$g(T^{\prime},t^{\prime}) := (\widetilde{T},\widetilde{t})$. Note that
$\widetilde{t}$ is a large tile in $\widetilde{T}$.  
Define $\widetilde{c}
(T^{\prime},t^{\prime}) = 2$. \vskip 2mm

Define a sequence of tilings $T'_0 = T', T'_1, T'_2, \dots$ and tiles $t'_0 = t', t'_1, t'_2, \dots$, recursively by
$(T'_{n+1}, t'_{n+1}) = g(T'_{n}, t'_{n})$.  Then define the {\bf code} of $(T',t')$ to be the binary string
\[
c(T^{\prime},t^{\prime}) = \widetilde{c}(T^{\prime}_0,t^{\prime}_0) \; \widetilde{c}
(T^{\prime}_1,t^{\prime}_1) \; \widetilde{c}(T^{\prime}_3,t^{\prime}_3) \cdots,
\]
and note that, as done in Section~\ref{sec:ca}, the amalgamation used to determine the code 
$c(T^{\prime},t^{\prime})$ of $T^{\prime}$ can be done combinatorially without knowing the homeomorphism.
\end{definition}

Theorem~\ref{thm:c} below states that $c(T^{\prime},t^{\prime})$
completely determines the original Ammann tiling $T$, and can be used to generate $T$.  
 An intuitive way to see this is to start with a patch
$P_{0}$ consisting of a single large Ammann tile $t$. Let $c(T,t) =
c_{1}\, c_{2}\, \cdots$. If $c_{1} = 1$, embed $P_{0}$ in a patch $P_{1}$ as
shown at the start of the bottom row in Figure~\ref{fig:nested}; if $c_{1} =
2$, embed $P_{0}$ in a patch $P_{1}$ as at the start of the top two rows in
Figure~\ref{fig:nested}. Continue in this way to get a nested sequence of
patches as in Figure~\ref{fig:nested}. The nested union is $T$.

\begin{remark}
\label{rem:hq} It is well known that there are a couple of special cases where the nested union $\cup \{P_n ; n\geq 0\}$ 
is a tiling that fills a half plane or a quadrant of the plane. Therefore, in this section such tilings will also be classified
as Ammann tilings. If $T$ is an Ammann tiling of the half plane, then a tiling
of the entire plane satisfying the matching rules can be obtained by
reflecting $T$ in the border line of $T$.  These tilings will not be classified
as Amman tilings in this section. Similarly for a tiling obtained from a
tiling $T$ of a quadrant by reflecting in the two border rays.  
\end{remark}

A more rigorous formulation of the above method of retrieving the original
tiling from the code uses the following combinatorial construction.
Recall that $s = 1/\tau$, where $\tau$ is the square root of the golden ratio. 
The fact that the golden bee $G$ is the non-overlapping union of two small
similar copies of itself can be stated more precisely as follows.
\[
G = f_{1}(G) \cup f_{2}(G),
\]
where $f_{1}(G)$ and  $f_{2}(G)$ are non-overlapping and 
\vskip 2mm
\begin{equation}
\label{eq:gb}f_{1}%
\begin{pmatrix}
x\\
y
\end{pmatrix}
=
\begin{pmatrix}
0 & -s\\
s & 0
\end{pmatrix}
\begin{pmatrix}
x\\
y
\end{pmatrix}
+%
\begin{pmatrix}
s\\
0
\end{pmatrix},\qquad f_{2}
\begin{pmatrix}
x\\
y
\end{pmatrix}
=
\begin{pmatrix}
s^{2} & 0\\
0 & -s^{2}%
\end{pmatrix}
\begin{pmatrix}
x\\
y
\end{pmatrix}
+%
\begin{pmatrix}
0\\
1
\end{pmatrix}
.
\end{equation}

\begin{figure}[htb]
\centering\includegraphics[width=\textwidth]{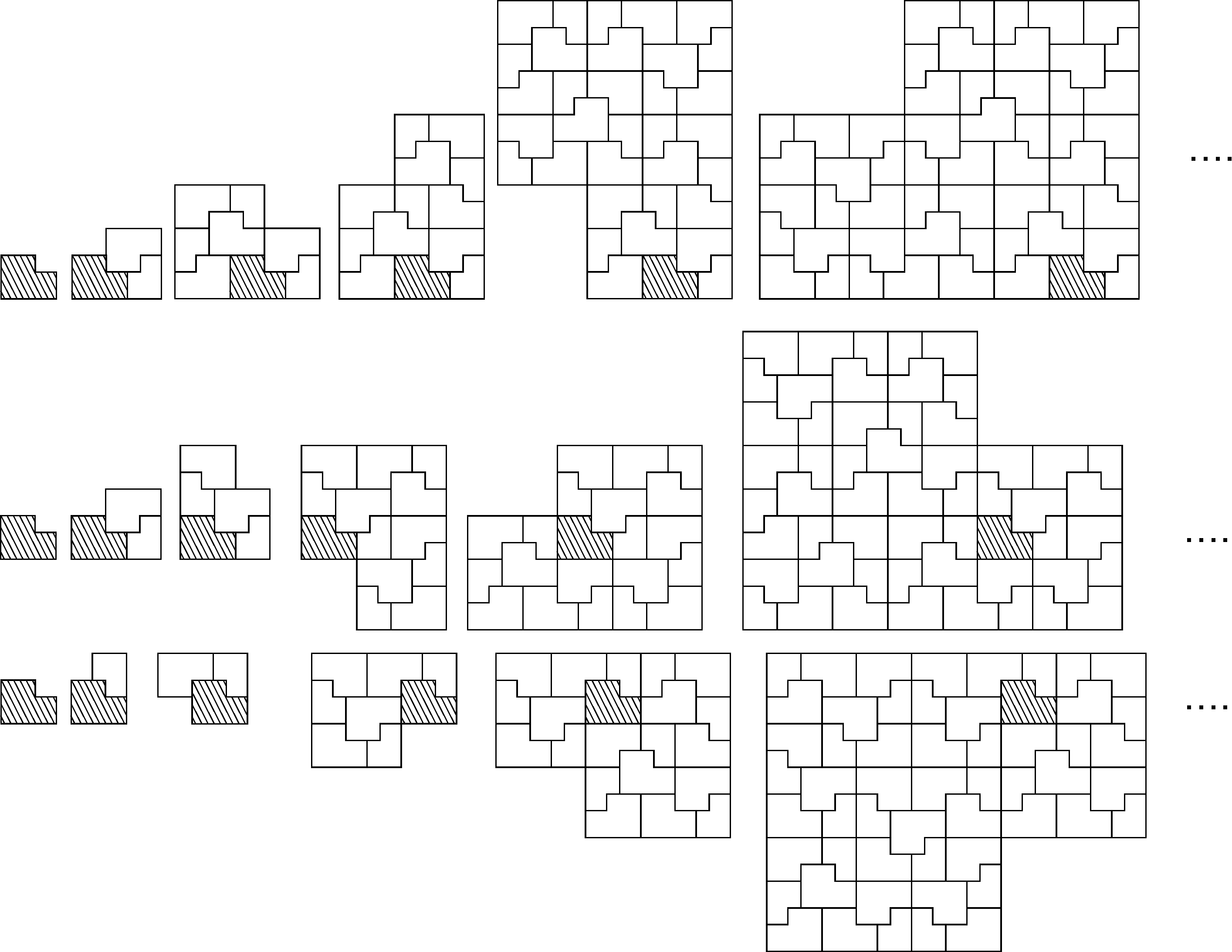}
\caption{The starts of three possible
sequences of embeddings, from left to right, as described in the text. The
initial large tile, the patch $P_{0}$, is shown shaded and has a fixed
location in the plane. }%
\label{fig:nested}%
\end{figure}

 For $\sigma= \sigma_{1}\, \sigma_{2}\, \dots\,
\sigma_{k} \in\{1,2\}^{*}$ let

\[
e(\sigma) := \sum_{i=1}^{k} \sigma_{i}%
\]
and let $\sigma_{*} := \sigma_{k}$ denote the last term in the string $\sigma$. We introduce the notation
\[
\begin{aligned} f_{\sigma} &:= f_{\sigma_1}\circ f_{\sigma_2 } \circ \cdots \circ f_{\sigma_k} \\ f_{-\sigma} &:= f_{\sigma_1}^{-1}\circ f_{\sigma_2 }^{-1} \circ \cdots \circ f_{\sigma_k}^{-1}. \end{aligned}
\]
and 
\[
\theta|k := \theta_{1} \, \theta_{2}\, \dots\theta_{k}%
\]
for $\theta= \theta_{1} \, \theta_{2}, \cdots\in\{1,2\}^{\infty}$. 
By convention, $\theta|0$ is the empty string $\emptyset$; $e(\emptyset) = 0$, and $f_{\emptyset}$ is
the identity.  

Let
\[
W(\theta, k) := \{ \sigma\in\{1,2\}^{*} \, :\, e(\sigma_{*}) > e(\sigma) -
e(\theta|k) \geq 0\}.
\]
Intuitively, the finite strings $\sigma \in W(\theta,k)$ are those whose sum differs from the sum of the elements of $\theta|k$
by zero or one.  This forces each tile $t(\theta,k,\sigma)$ in Definition~\ref{def:3} below to be congruent to either a large
or small Ammann tile.  

\begin{definition} \label{def:3}
For each $\theta\in\{1,2\}^{\infty}$, a tiling $T(\theta)$ is constructed in
three steps. \vskip 2mm
\begin{enumerate}
\item \textit{A single tile.} For each integer $k\geq1$ and each $\sigma
\in\lbrack N]^{\ast}$, construct a single tile $t(\theta,k,\sigma)$ that is
similar to $G$:
\[
t(\theta,k,\sigma):=(f_{-(\theta\,|\,k)}\circ f_{\sigma})(sG).
\]
\vskip 1mm

\item \textit{A patch of tiles.} Form a patch $T(\theta,k)$ of tiles given
by:
\[
T(\theta,k):=\left\{  t(\theta,k,\sigma)\,:\, \sigma\in W(\theta,k) \right\}.
\]
\vskip 1mm

\item \textit{A tiling.} The tiling ${T}(\theta)$, depending only on
$\theta\in\{1,2\}^{\infty}$, is the union of the patches $T(\theta,k)$, which is known \cite{polygon}
to be a nested union:
\[
{T}(\theta):=\bigcup_{k\geq1}T(\theta,k).
\]
\end{enumerate}

\noindent The set $\{1,2\}^{\infty}$ is called the {\bf parameter set}. For each
parameter $\theta$, the tiling $T(\theta)$ is called the $\mathbf \theta
$-{\bf tiling}.  Modulo Remark~\ref{rem:hq}, a $\theta$-tiling is 
an Ammann tiling. 
\end{definition} 

\begin{theorem} \label{thm:c} 
Let $T^{\prime}= h(T)$ be the image of Ammann tiling $T$ under a
homeomorphism $h$ of the plane.  Let $t^{\prime}$ be any large tile in
$T^{\prime}$, and let $\theta:= c(T^{\prime},t^{\prime})$ be the code of
$(T^{\prime},t^{\prime})$. Then $T = T(\theta)$, independent of which large
tile $t^{\prime}\in T^{\prime}$ is chosen.
\end{theorem}

\begin{proof}
Let $t = h^{-1}(t^{\prime}) \in T$, and note
that $c(T^{\prime},t^{\prime}) = c(T,t)$. Let $\theta := c(T, t)$.  It must be shown that $T(\theta) = T$ 
for any large tile $t\in T$.  Since an Ammann tiling is determined by its code, it suffices to show that, 
for any large tile $t\in T$, there
is a large tile $p\in T(\theta)$ such that  that $c(T(\theta), p) = \theta$.    

A tile $t(\theta,k,\sigma) \in T(\theta)$ is a large tile congruent to $sG$ in $T(\theta)$ if $e(\sigma) -e(\theta\,|\,K)= 0$ and a small tile congruent to $s^2 G$ if $e (\sigma) -e(\theta \,|\,K)= 1$.  From the definitions, this implies that
 $t(\theta,k,\sigma) \in T(\theta)$ is a small tile if and only if $\sigma_* = 2$, i.e., $\sigma = 
\sigma_1 \, \sigma_2 \cdots \sigma_{j-1} \, 2$, and the tile $t(\theta,k,\omega)$,
 where $\omega = \sigma_1 \, \sigma_2 \cdots \sigma_{j-1} \, 1$, is the large partner of $t(\theta,k,\sigma)$.  
Let $p = t(\theta,k,\sigma)$, where $\sigma = \sigma_1 \, \sigma_2 \cdots \sigma_j$, be an arbitrary tile in $T(\theta)$.  
Then winding through
Definition~\ref{def:3}, the code of $c(T(\theta), p)$ is given by $c(T(\theta), p) =\sigma_{j}\,\sigma_{j-1} \cdots\sigma_{1}\,\theta_{k+1}\,\theta_{k+2}\cdots$.

Consider two cases.
If the original large tile $t\in T$ has a partner, then $\theta_1 = 1$, i.e., $c(T,t) = 1 \theta_2 \, \theta_3 \cdots$.  In $T(\theta)$, let $p = t(\theta,1,1)$.  
Then the code of $c(T(\theta),p)$ is given by $c(T(\theta), p) = 1 \theta_2 \, \theta_3 \cdots$.
If the original large tile $t\in T$ has no partner, then $\theta_1 = 2$, i.e., $c(T,t) = 2 \theta_2 \, \theta_3 \cdots$.  
In $T(\theta)$, let $p = t(\theta,1,2)$.  
Then the code of $c(T(\theta),p)$ is given by $c(T(\theta), p) = 2 \theta_2 \, \theta_3 \cdots$.  In either case
$c(T(\theta), p) = \theta$.
\end{proof}

\begin{remark}
An arbitrarily large patch of the original tiling $T$ can be obtained in finite time and efficiently from the distorted tiling $T'$ by first using the combinatorial
 amalgamation Algorithm B to obtain a finite concatenation $\theta|k$ of a code $\theta$ for $T'$ and then by using the procedure in Definition~\ref{def:3} to construct the patch $T(\theta,k)$.  
\end{remark}

\section*{aknowlegements}
We thank Vanessa Robbins for interesting conversations in relation to this work.

\end{document}